\documentclass[12pt]{article}
\usepackage{latexsym,amssymb,float,amsmath,amsthm,enumerate,cite,geometry}
\newtheorem{theorem}{Theorem}
\newtheorem{proposition}[theorem]{Proposition}
\newtheorem{conjecture}[theorem]{Conjecture}

\newtheorem{claim}{Claim}
\usepackage{lineno}
\usepackage{setspace}
\allowdisplaybreaks

\usepackage{tikz}
\tikzset{
dot/.style = {circle, fill, minimum size=#1,
inner sep=0pt, outer sep=0pt},
dot/.default = 4pt
}

\begin{document}
\onehalfspace

\title{Forest Cuts in Sparse Graphs}
\author{Vsevolod Chernyshev$^1$\and 
Johannes Rauch$^2$ \and 
Dieter Rautenbach$^2$}
\date{}

\maketitle
\vspace{-10mm}
\begin{center}
$^1$ 
Universit\'{e} Clermont-Auvergne, CNRS, Clermont-Auvergne-INP, LIMOS\\ 
63000 Clermont-Ferrand, France\\
\texttt{vchern@gmail.com}\\[5mm]
$^2$ Ulm University, Institute of Optimization and Operations Research\\ 
89081 Ulm, Germany\\
\texttt{$\{$johannes.rauch,dieter.rautenbach$\}$@uni-ulm.de}
\end{center}

\begin{abstract}
We propose the conjecture that 
every graph $G$ of order $n$ with less than $3n-6$ edges
has a vertex cut that induces a forest.
Maximal planar graphs do not have such vertex cuts and 
show that the density condition would be best possible.
We verify the conjecture for planar graphs
and show that 
every graph $G$ of order $n$ with less than $\frac{11}{5}n-\frac{18}{5}$ edges
has a vertex cut that induces a forest.\\[3mm]
\end{abstract}

\section{Introduction}\label{sec1}

We consider finite, simple, and undirected graphs and use standard terminology.

Confirming a conjecture due to Caro, 
it was shown by Chen and Yu \cite{chyu} 
that every graph $G$ of order $n$ with less than $2n-3$ edges 
has an independent cut, 
which is a vertex cut that is an independent set.
In the present paper we consider forests cuts,
which are vertex cuts inducing a forest. 
Our main objective is an analogue of the result of Chen and Yu.
More precisely, we ask which density condition 
implies the existence of a forest cut and pose the following conjecture.

\begin{conjecture}\label{conjecture1}
If $G$ is a graph of order $n$ with less than $3n-6$ edges, 
then $G$ has a forest cut.
\end{conjecture}
As a natural approach to Conjecture \ref{conjecture1}
we tried to adapt the elegant inductive proof of Chen and Yu from \cite{chyu},
which hinges on the following statement. 

\begin{theorem}[Chen and Yu, Theorem 1 in \cite{chyu}]\label{theorem1}
If $G$ is a $2$-connected graph of order $n$ with less than $2n-3$ edges,
then, for every vertex $u$ of $G$,
there is an independent cut $S$ of $G$ with $u\not\in S$.
\end{theorem}

Unfortunately, the obvious analogue of this statement is false.
In fact, the graph 
that arises from a cycle $x_1x_2\ldots x_{2k}$ of order $2k$
by adding the $k$ edges $x_1x_{k+1},x_2x_{k+2},\ldots,x_kx_{2k}$
and adding one universal vertex $y$
is $4$-connected of order $n=2k+1$
with only $\frac{5}{2}(n-1)$ edges 
and the universal vertex $y$ necessarily belongs to every vertex cut.
Furthermore, the proof of Theorem \ref{theorem1} in \cite{chyu}
relies on contractions. 
Adapting this contraction argument to the context of forest cuts
can lead to universal vertices
even though the initial graph had no universal vertex.

Nevertheless, universal vertices do not seem to be a fundamental problem
for Conjecture \ref{conjecture1}:
If $G$ is a graph of order $n$ with less than $3n-6$ edges 
that contains a universal vertex $u$, then
$G-u$ has order $n-1$ and less than $3n-6-(n-1)=2(n-1)-3$ edges.
By the result of Chen and Yu, the graph $G-u$ has an independent cut $S$.
Now, the set $\{ u\}\cup S$ is a forest cut of $G$
and we obtain the following.

\begin{proposition}\label{proposition0}
Conjecture \ref{conjecture1} holds for graphs that contain a universal vertex.
\end{proposition}
Not only universal vertices necessarily belong to every forest cut;
the only forest cut of the $3$-connected graph 
that arises from $K_{3,n-3}$ with $n>6$
by adding a cycle of length less than $n-3$ 
in the partite set of order $n-3$
is the other partite set of order $3$.

It is well known that $3n-6$ is exactly the number of edges 
of maximal planar graphs of order $n$
and it is not difficult to show Conjecture \ref{conjecture1} for planar graphs.
Maximal planar graphs also show that 
Conjecture \ref{conjecture1} would be best possible.
We postpone proofs to Section \ref{section2}.

\begin{proposition}\label{proposition1}
Conjecture \ref{conjecture1} is true for planar graphs.
Furthermore, no maximal planar graph has a forest cut.
\end{proposition}

Le and Pfender \cite{lepf} showed that 
all graphs $G$ of order $n$ with exactly $2n-3$ edges 
that do not have an independent cut 
arise from the triangle $K_3$ and the triangular prism $\bar{C}_6$
by recursively glueing graphs along edges $K_2$ or triangles $K_3$.
Similarly, the graphs that arise from maximally planar graphs
by recursively glueing graphs along $K_3$s or $K_4$s
are further examples 
showing that Conjecture \ref{conjecture1} would be best possible.

Our main contribution is the following result, 
whose proof is based on local considerations and linear programming.

\begin{theorem}\label{theorem2}
If a graph $G$ of order $n$ has less than $\frac{11}{5}n-\frac{18}{5}$ edges,
then $G$ has a forest cut.
\end{theorem}
Before we proceed to the proofs in Section \ref{section2},
we mention some related research.
Independent cuts were first studied by Tucker~\cite{tu}.
Answering a question posed by Corneil and Fonlupt~\cite{cofo}
and using a result of Chv\'{a}tal~\cite{ch},
Klein and de Figueiredo showed that 
deciding the existence of an independent cut is NP-complete.
This initial hardness result was strengthened in different ways
and also some tractable cases were identified~\cite{BrDrLeSz,LeRa,le_mosca_mueller_2008}.
The parameterized complexity was studied in~\cite{MaOsRa,raraso,krle}.
Also small independent cuts~\cite{MR1936948} 
and cuts inducing graphs of small maximum degree~\cite{beraraso}
were studied.

The hardness of deciding the existence of a forest cut 
follows immediately from the following simple observation:
If a graph $G$ arises from some graph $H$ by adding a universal vertex,
then $G$ has a forest cut if and only if $H$ has an independent cut.

\section{Proofs}\label{section2}

We briefly collect some notation.
Let $G$ be a graph and let $S$ be a set of vertices of $G$.
The subgraph of $G$ induced by $S$ is denoted by $G[S]$.
The set $S$ is a vertex cut of $G$ if $G-S$ is disconnected.
A vertex cut $S$ of $G$ is an {\it independent cut} if $S$ is independent,
that is, the graph $G[S]$ has no edges.
A vertex cut $S$ of $G$ is a {\it forest cut} if $G[S]$ is a forest.
A vertex of a graph $G$ is {\it universal} in $G$
if it is adjacent to all other vertices of $G$.

\begin{proof}[Proof of Proposition \ref{proposition1}]
In order to prove Conjecture \ref{conjecture1} for planar graphs
it suffices to show that $G-xy$ has a forest cut
where $G$ is a maximal planar graph and $xy$ is some edge of $G$.
Let $G$ be embedded in the plane such that 
$xyzx$ is the boundary of the unbounded face.
If $x$ and $y$ are the only neighbors of $z$,
then $G$ is the triangle $xyzx$ and $\{ z\}$ is a forest cut of $G-xy$.
Now, let $z$ have more than two neighbors.
Let $x,u_1,u_2,\ldots,u_k,y$ be the neighbors of $z$ 
as they appear in cyclic order around $z$ within the embedding.
Since $G$ is a triangulation, 
it follows that $P:xu_1u_2\ldots u_ky$ is a path in $G-xy$.
Let $Q:xu_{i_1}u_{i_2}\ldots u_{i_\ell}y$ be a path in $G[V(P)]-xy$
such that $i_1<i_2<\ldots<i_\ell$ and $\ell$ is minimum.
If the embedding of $G$ has vertices in the interior of the cycle $Q+xy$,
then $V(Q)$ is a forest cut of $G-xy$;
otherwise $\{ z,u_{i_1}\}$ is a forest cut of $G-xy$.

For the second part of the statement,
let $G$ again be a maximal planar graph
and fix a plane embedding of $G$, 
which is a plane triangulation.
Suppose, for a contradiction, that $G$ has a forest cut $S$.
By choosing $S$ inclusionwise minimal, 
we may assume that every vertex in $S$ 
has neighbors in different components of $G-S$.
Let $u$ be a vertex of degree at most one in $G[S]$.
There are two neighbors $v$ and $w$ of $u$ in $G$
that belong to different components of $G-S$
such that the two edges $uv$ and $uw$ 
are embedded in a cyclically consecutive way at $u$.
Since the embedding is a plane triangulation,
the vertices $v$ and $w$ are adjacent,
contradicting that they belong to different components of $G-S$.
\end{proof}

\begin{proof}[Proof of Theorem \ref{theorem2}]
Suppose, for a contradiction, that $G$ is a counterexample of minimum order,
that is, the graph $G$ of order $n$ has less than $\frac{11}{5}n-\frac{18}{5}$ edges
but no forest cut.
Since the statement is trivial for $n\leq 3$, it follows that $n\geq 4$.
Since every vertex cut of order at most $2$ is a forest cut, 
it follows that $G$ is $3$-connected.
For $n\in \{ 4,5\}$, 
there are no $3$-connected graphs with less than 
$\frac{11}{5}n-\frac{18}{5}\in \{ 5.2,7.4\}$ edges.
For $n\in \{ 6,7\}$,
all $3$-connected graphs with less than 
$\frac{11}{5}n-\frac{18}{5}\in \{ 9.6,11.8\}$ edges 
are shown in Figure \ref{fig:n67};
all these graphs have forest cuts.
We generated these four graphs using \cite{codhgo}.

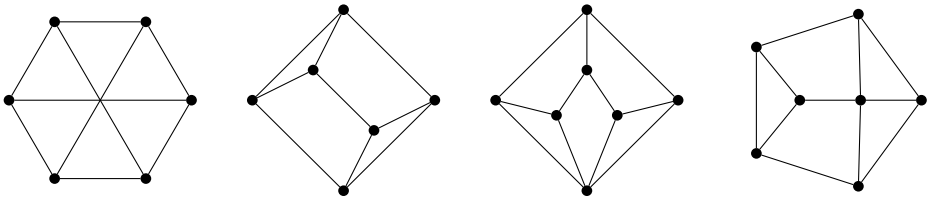
\begin{figure}[H]
\centering
\begin{tikzpicture}[scale=0.8]
\begin{scope}[shift={(0,0)}]
\foreach \v in {1, ..., 6} {
	\node[dot] (\v) at (-\v*60:1.5) {};
};
\draw (1) to (2) to (3) to (4) to (5) to (6) to (1);
\draw (1) to (4);
\draw (2) to (5);
\draw (3) to (6);
\end{scope}
\begin{scope}[shift={(4,0)}]
\node[dot] (1) at (0,1.5) {};
\node[dot] (2) at (1.5,0) {};
\node[dot] (3) at (0,-1.5) {};
\node[dot] (4) at (-1.5,0) {};
\node[dot] (5) at (-0.5,0.5) {};
\node[dot] (6) at (0.5,-0.5) {};
\draw (1) to (2) to (3) to (4) to (1);
\draw (1) to (5) to (4);
\draw (2) to (6) to (3);
\draw (5) to (6);
\end{scope}
\begin{scope}[shift={(8,0)}]
\node[dot] (1) at (0,1.5) {};
\node[dot] (2) at (1.5,0) {};
\node[dot] (3) at (0,-1.5) {};
\node[dot] (4) at (-1.5,0) {};
\node[dot] (5) at (0,0.5) {};
\node[dot] (6) at (-0.5,-0.25) {};
\node[dot] (7) at (0.5,-0.25) {};
\draw (1) to (2) to (3) to (4) to (1);
\draw (3) to (6) to (4);
\draw (3) to (7) to (2);
\draw (1) to (5) to (6);
\draw (5) to (7);
\end{scope}
\begin{scope}[shift={(12,0)}]
\foreach \v in {1, ..., 5} {
	\node[dot] (\v) at (-\v*72:1.5) {};
};
\node[dot] (6) at (-0.5,0) {};
\node[dot] (7) at (0.5,0) {};
\draw (1) to (2) to (3) to (4) to (5) to (1);
\draw (1) to (7) to (5);
\draw (4) to (7) to (6) to (3);
\draw (2) to (6);
\end{scope}
\end{tikzpicture}
\caption{All 3-connected graphs on $n \in \{6, 7\}$ vertices 
with less than $\frac{11}{5}n-\frac{18}{5}$ edges.}\label{fig:n67}
\end{figure}

It follows that $n\geq 8$.

\begin{claim}\label{claim4conn}
The graph $G$ is $4$-connected.
\end{claim}
\begin{proof}[Proof of Claim \ref{claim4conn}]
Suppose, for a contradiction, that $G$ has a vertex cut $S$ of order $3$.
Since $G$ has no forest cut, the set $S$ induces a triangle.
Let $G_1$ and $G_2$ be 3-connected proper subgraphs of $G$ of orders $n_1$ and $n_2$, respectively,
such that $G_1\cup G_2=G$ and $V(G_1)\cap V(G_2)=S$.
Since a forest cut of $G_1$ or $G_2$ is also a forest cut of $G$,
it follows that neither $G_1$ nor $G_2$ has a forest cut and 
the choice of $G$ implies that each $G_i$ has at least $\frac{11}{5}n_i-\frac{18}{5}$ edges.
Since $n=n_1+n_2-3$ and $G_1$ and $G_2$ share exactly three edges, 
we obtain the contradiction
\begin{eqnarray*}
m(G) & = & m(G_1)+m(G_2)-3
\geq \left(\frac{11}{5}n_1-\frac{18}{5}\right)+\left(\frac{11}{5}n_2-\frac{18}{5}\right)-3
= \frac{11}{5}n-\frac{18}{5},
\end{eqnarray*}
which completes the proof of the claim.
\end{proof}

\begin{claim}\label{claim4}
If $u$ is a vertex of $G$ of degree $4$,
then at most two neighbors of $u$ have degree $4$.
\end{claim}
\begin{proof}[Proof of Claim \ref{claim4}]
Suppose, for a contradiction, 
that a vertex $u$ of degree $4$ has three neighbors $a$, $b$, and $c$ of degree $4$.
Let $d$ be the fourth neighbor of $u$.
Since $G$ is $4$-connected and $n\geq 8$,
Hall's theorem implies that 
the vertices $a$, $b$, and $c$
have distinct neighbors $a'$, $b'$, and $c'$, respectively,
outside of $N_G[u]$.

First, suppose that $abca$ is a triangle in $G$.
In this case, $\{ a',b,c,d\}$ is a forest cut, which is a contradiction.
Next, suppose that $abda$ is a triangle in $G$.
Since $\{ a',b,c,d\}$ is not a forest cut, 
it follows that $a'cda'$ is a triangle.
Note that $c$ is not adjacent to $b'$.
Now, the set $\{ a,b',c,d\}$ is a forest cut, which is a contradiction.
Similarly, it follows that $acda$ and $bcdb$ are not triangles in $G$.
Since $\{ a,b,c,d\}$ is not a forest cut, 
we may assume that $abcda$ is a cycle in $G$.
This implies that $\{ a',b,c,d\}$ is a forest cut, which is a contradiction
and completes the proof of the claim.
\end{proof}

\begin{claim}\label{claim3}
If $u$ is a vertex of $G$ of degree $5$,
then not all neighbors of $u$ have degree $4$.
\end{claim}
\begin{proof}[Proof of Claim \ref{claim3}]
Suppose, for a contradiction, that all neighbors of $u$ have degree $4$.
Let $N'$ be the set of neighbors $v$ of $u$ with $N_G(v)\not\subseteq N_G[u]$.
Note that $N'$ is a vertex cut of $G$.
Let $N''$ be the set of vertices that lie on a cycle in $G[N']$.
Since every vertex in $N'$ has degree $4$ and is adjacent to $u$ 
as well as to some vertex outside of $N_G[u]$,
the graph $G[N'']$ is the union of disjoint cycles
and there is no edge between $N''$ and $N'\setminus N''$.
Since $|N''|\leq 5$, it follows that $G[N'']$ consists of a single cycle $C$.
Let $X$ be the set of neighbors of vertices on $C$ outside of $N_G[u]$.
For $x\in X$,
let $S_x=\{ x\}\cup (N''\setminus N_G(x))\cup (N'\setminus N'')$.
Since $S_x$ is not a forest cut of $G$,
the structure of $G$ implies 
that $\{ x\}\cup (N'\setminus N'')$ contains a cycle.
Since $|N'\setminus N''|\leq 2$,
it follows that $N'\setminus N''$ consists of two adjacent vertices,
$n(C)=3$, and $x$ is adjacent to every vertex in $N'\setminus N''$,
that is, the set $\{ x\}\cup (N'\setminus N'')$ induces a triangle.
By symmetry, we obtain that 
every vertex in $X$ is adjacent to every vertex in $N'\setminus N''$.
Since the vertices in $N'\setminus N''$ have degree $4$,
this implies $|X|\leq 2$.
If $|X|<2$, 
then $n\geq 8$ implies that $X\cup (N'\setminus N'')$ is a vertex cut of $G$,
contradicting the fact that $G$ is $4$-connected.
Hence, we obtain that $|X|=2$.
If $n>8$, 
then $X$ is a vertex cut of $G$,
again contradicting the fact that $G$ is $4$-connected.
Hence, we obtain that $n=8$.
Now, the set $\{ u\}\cup X$ is a forest cut of $G$,
which is a contradiction 
and completes the proof of the claim.
\end{proof}
For a set $S$ of vertices, let $d_G(S)=\sum\limits_{u\in S}d_G(u)$.

\begin{claim}\label{lem:c3}
If $u$ is a vertex of $G$ of degree $4$ and $N_G(u)$ contains a triangle,
then $d_G(N_G(u)) \geq 19$.
\end{claim}
\begin{proof}[Proof of Claim \ref{lem:c3}]
Let $N_G(u)=\{ a,b,c,d\}$ and let $abca$ be a triangle in $G$.
If $u$ has only one neighbor of degree $4$, then $d_G(N_G(u)) \geq 4+5+5+5=19$.
If $u$ has a neighbor of degree at least $6$, then, by Claim \ref{claim4},
we obtain $d_G(N_G(u)) \geq 4+4+5+6=19$.
Hence, we may assume that two neighbors of $u$ have degree $4$
and the remaining two neighbors of $u$ have degree $5$.

First, suppose that $a$ and $b$ are the two neighbors of $u$ of degree $4$.
Since $G$ is $4$-connected, 
it follows that the vertices $a$ and $b$ have distinct neighbors $a'$ and $b'$,
respectively, outside of $N_G[u]$.
Since $\{ a',b,c,d\}$ is not a forest cut, 
it follows that $a'cda'$ is a triangle.
Since the vertex $c$ has degree $5$,
it follows that $c$ is not adjacent to $b'$.
Now, the set $\{ a,b',c,d\}$ is a forest cut, 
which is a contradiction.
Hence, the triangle $abca$ contains only one vertex of degree $4$.
By symmetry, 
we may assume that $a$ and $d$ are the two neighbors of $u$ of degree $4$.
Since $G$ is $4$-connected, 
it follows that the vertex $a$ has a neighbor $a'$ outside of $N_G[u]$.

We consider some cases.

\medskip

\noindent {\bf Case 1} {\it $bd,cd\in E(G)$.}

\medskip

\noindent Let $d'$ be the neighbor of $d$ distinct from $u$, $b$, and $c$.
Since $G$ is $4$-connected and $n \geq 8$, it follows that $a'$ and $d'$ are distinct.
Since $\{ a',b,c,d'\}$ is not a forest cut, 
the set $\{ a',b,c,d'\}$ contains a cycle of length $4$.
Since $b$ and $c$ have degree $5$, these two vertices have no neighbor 
outside of $N_G[u]\cup \{ a',d'\}$.
Now, the set $\{ a',d'\}$ is a vertex cut, which is a contradiction.

\medskip

\noindent {\bf Case 2} {\it $a'c,cd\in E(G)$.}

\medskip

\noindent Since the vertex $c$ has degree $5$,
it follows that $\{ a',b,d\}$ is a vertex cut,
which contradicts Claim \ref{claim4conn}.

\medskip

\noindent {\bf Case 3} {\it $a'b,bd\in E(G)$.}

\medskip

\noindent This case is symmetric to Case 2.

\medskip

\noindent {\bf Case 4} {\it $a'b,a'c\in E(G)$.}

\medskip

\noindent Since $G$ is $4$-connected and $n\geq 8$,
it follows that $b$ and $c$ have distinct neighbors $b'$ and $c'$,
respectively, outside of $N_G[u]\cup \{ a'\}$.
Since $\{ a',b',c,d\}$ is not a forest cut, 
it follows that $a'b'da'$ is a triangle.
Since $\{ a',b,c',d\}$ is not a forest cut, 
it follows that $a'c'da'$ is a triangle.
Since $G$ has more than $14$ edges, it follows that $n>8$.
Now, the set $\{ a',b',c'\}$ is a vertex cut,
which is a contradiction.

\medskip

\noindent Since $\{ a',b,c,d\}$ is not a forest cut, 
it follows that this set contains a cycle.
Since all four cases led to contradictions,
it follows that $\{ a',b,c,d\}$ contains no triangle.
By symmetry, we may assume that $G[\{ a',b,c,d\}]$
is the cycle $a'bcda'$.
Let $d'$ be the neighbor of $d$ distinct from $u$, $a'$, and $c$.
Since $G$ is $4$-connected and $n\geq 8$, 
it follows that 
$b$ is not adjacent to $d'$ and that 
$c$ is not adjacent to $a'$ or $d'$.
Now, the set $\{ a',b,c,d'\}$ is a forest cut,
which is a contradiction and completes the proof.
\end{proof}

\begin{claim}\label{lem:c4}
If $u$ is a vertex of $G$ of degree $4$ and $N_G(u)$ contains no triangle,
then $d_G(N_G(u)) \geq 19$.
\end{claim}
\begin{proof}[Proof of Claim \ref{lem:c4}]
Since $N_G(u)$ is not a forest cut,
we may assume, by Claim \ref{lem:c3},
that $G[N_G(u)]$ is the cycle $abcda$.
Similarly as in the proof of Claim \ref{lem:c3},
we may assume that two neighbors of $u$ have degree $4$
and the remaining two neighbors of $u$ have degree $5$.

First, suppose that $a$ and $b$ are the two neighbors of $u$ of degree $4$.
Since $G$ is $4$-connected, 
it follows that $a$ and $b$ have distinct neighbors $a'$ and $b'$,
respectively, outside of $N_G[u]$.
Since $\{ a',b,c,d\}$ is not a forest cut,
it follows that $a'$ is adjacent to $c$ and $d$.
Since $\{ a,b',c,d\}$ is not a forest cut,
it follows that $b'$ is adjacent to $c$ and $d$.
Since $c$ and $d$ have degree $5$ and $n \geq 8$,
it follows that $\{ a',b'\}$ is a vertex cut,
which is a contradiction.
Hence, we may assume that $a$ and $c$ are the two neighbors of $u$ of degree $4$.
Since $G$ is $4$-connected, 
it follows that $a$ and $c$ have distinct neighbors $a'$ and $c'$,
respectively, outside of $N_G[u]$.
Since $\{ a',b,c,d\}$ is not a forest cut,
it follows that $a'$ is adjacent to $b$ and $d$.
Since $\{ a,b,c',d\}$ is not a forest cut,
it follows that $c'$ is adjacent to $b$ and $d$.
Since $b$ and $d$ have degree $5$,
it follows that $\{ a',c'\}$ is a vertex cut,
which is a contradiction and completes the proof of the claim.
\end{proof}

For $i \in \{4, \ldots, n-1\}$, 
let $V_i$ be the set of vertices of degree $i$.
For $j \in \{5, \ldots, n-1\}$, 
let $V_4^j$ be the set of vertices $u$ of degree $4$ 
with $j=\max\{d_G(v):v \in N_G(u)\}$.
Let $V_4^6{'}$ be the set of vertices in $V_4^6$ 
that have exactly one neighbor of degree 6.
Let $V_4^6{''} = V_4^6 \setminus V_4^6{'}$.
Let 
$n_i=|V_i|$,
$n_4^j=|V_4^j|$,
$n_4^6{'}=|V_4^6{'}|$, and 
$n_4^6{''}=|V_4^6{''}|$.

By Claim \ref{lem:c3} and Claim \ref{lem:c4},
it follows that $d_G(N_G(u))\geq 19$ for every vertex $u$ of degree $4$.
This implies that every vertex in $V_4^5$ has at least three neighbors in $V_5$
and that every vertex in $V_4^6{'}$ has at least one neighbor in $V_5$.
By Claim \ref{claim3}, we obtain that 
$4n_5\geq 3n_4^5+n_4^6{'}$.
The definitions of $V_4^6{'}$ and $V_4^6{''}$ imply
$n_4^6=n_4^6{'}+n_4^6{''}$ and
$6n_6\geq n_4^6{'}+2n_4^6{''}$.
For $j \in \{7, \ldots, n-1\}$, 
the definition of $V_4^j$ implies 
$jn_j\geq n_4^j$.
By Claim \ref{claim4}, it follows that 
$\sum\limits_{j=5}^{n-1}jn_j\geq 2n_4$.

It follows that the number of edges of $G$ 
is at least the optimal value ${\rm OPT}(P)$ of the linear program $(P)$
shown below.

The following claim establishes our final contradiction.
\begin{claim}\label{claim5}
${\rm OPT}(P)\geq \frac{11}{5}n$.
\end{claim}
\begin{proof}[Proof of Claim \ref{claim5}]
In order to formulate the dual linear program $(D)$ for $(P)$,
we introduce dual variables:
For the first four constraints in $(P)$, 
the dual variables are $x_1$, $x_2$, $x_3$, and $x_4$,
respectively, where $x_4$ is required to be non-negative, 
because the fourth constraint is an inequality.
For the next two constraints, 
the dual variables are $y_5$ and $y_6$,
which are both required to be non-negative.
Finally, for the constraint $jn_j-n_4^j\geq 0$ with $j\in \{ 7,\ldots,n-1\}$,
we introduce the non-negative dual variable $y_j$.
With these dual variables, the dual linear program $(D)$ is as shown below.

$$
(P)\hspace{2cm}\begin{array}{lrlll}
\mbox{minimize} & \frac{1}{2}\sum\limits_{i=4}^{n-1}i n_i & & &\\[4mm]
\mbox{subject to} & \sum\limits_{i=4}^{n-1}n_i & = & n &\\
				& n_4-\sum\limits_{j=5}^{n-1}n_4^j & = & 0 &\\
				& n_4^6-{n_4^6}'-{n_4^6}'' & = & 0 &\\
				& \sum\limits_{j=5}^{n-1}jn_j-2n_4 & \geq & 0 &\\[4mm]
				& 4n_5-3n_4^5-{n_4^6}' & \geq & 0 &\\
				& 6n_6-{n_4^6}'-2{n_4^6}'' & \geq & 0 &\\
				& jn_j-n_4^j & \geq & 0 & \mbox{ for all }j\in \{ 7,\ldots,n-1\}\\[4mm]
				& n_i & \geq & 0 & \mbox{ for all }i\in \{ 4,\ldots,n-1\}\\
				& n_4^j & \geq & 0 & \mbox{ for all }i\in \{ 5,\ldots,n-1\}\\
				& {n_4^6}',{n_4^6}'' & \geq & 0. &
\end{array}
$$

$$
(D)\hspace{2cm}\begin{array}{lrlll}
\mbox{maximize} & nx_1 & & &\\
\mbox{subject to} & x_2-2x_4+x_1 & \leq & 2 &\\
				& 4y_5+5x_4+x_1 & \leq & \frac{5}{2} &\\
				& 6y_6+6x_4+x_1 & \leq & 3 &\\
				& jy_j+jx_4+x_1 & \leq & \frac{j}{2} & \mbox{ for all }j\in \{ 7,\ldots,n-1\}\\[4mm]
				& -x_2-3y_5 & \leq & 0 &\\
				& -x_2+x_3 & \leq & 0 &\\
				& -x_2-y_j & \leq & 0 & \mbox{ for all }j\in \{ 7,\ldots,n-1\}\\
				& -y_5-y_6-x_3 & \leq & 0 &\\
				& -2y_6-x_3 & \leq & 0 &\\[4mm]
				& x_1,x_2,x_3 & \in & \mathbb{R} &\\
				& x_4 & \geq & 0 &\\
				& y_j & \geq & 0 & \mbox{ for all }j\in \{ 5,\ldots,n-1\}.
\end{array}
$$
Since 
$x_1=\frac{11}{5}$,
$x_2=x_3=-\frac{6}{35}$,
$x_4=\frac{1}{70}$,
$y_5=\frac{2}{35}$,
$y_6=\frac{4}{35}$, and
$y_j=\frac{6}{35}$ for $j\in \{ 7,\ldots,n-1\}$
is a feasible solution for $(D)$,
weak duality implies 
${\rm OPT}(P)\geq {\rm OPT}(D)\geq nx_1=\frac{11}{5}n$,
which completes the proof.
\end{proof}
As observed above, this completes the proof of Theorem~\ref{theorem2}.
\end{proof}

\section{Conclusion}

We mention some problems for further research.

Clearly, the main problem motivated by the current paper is Conjecture \ref{conjecture1}.
Provided the conjecture is true,
it might be possible to constructively characterize the graphs of order $n$
with $3n-6$ edges that do not have a forest cut
similarly as done by Le and Pfender \cite{lepf}
for graphs without an independent cut.
Finally, we propose one further open problem.
Every graph $G$ that has no forest cut is necessarily $3$-connected
and, for every vertex $u$ of $G$, the subgraph $G[N_G(u)]$ contains
a cycle. 
For $k\in\mathbb{N}$,
let $G_k$ be the graph that arises from 
the cycle $u_0u_1\ldots u_{3k+2}u_0$ 
by adding the edges 
$u_0u_2,u_3u_5,\ldots,u_{3k}u_{3k+2}$
and by adding one further universal vertex.
The graph $G_k$ 
has $3k+4$ vertices and $7k+7$ edges, 
is $3$-connected and
every neighborhood of a vertex contains a cycle.
We believe that there is no such graph with less edges
and pose the following.

\begin{conjecture}\label{conjecture2}
If $G$ is a graph of order $n$ and size $m$
such that $G$ is $3$-connected and
every neighborhood of a vertex contains a cycle,
then $m\geq \frac{7}{3}(n-1)$.
\end{conjecture}

\noindent {\bf Acknowledgement} This work was partially supported 
by the International Research Center  
``Innovation Transportation and Production Systems'' of the I-SITE CAP 20-25
and by the Deutsche Forschungsgemeinschaft 
(DFG, German Research Foundation) -- project number 545935699.


\begin{thebibliography}{}
\bibitem{beraraso} S. Bessy, J. Rauch, D. Rautenbach, and U.S. Souza, Sparse vertex cutsets and the maximum degree, arXiv:2304.10353.
\bibitem{BrDrLeSz} A. Brandst\"{a}dt, F.F. Dragan, V.B. Le, and T. Szymczak, On stable cutsets in graphs, Discrete Applied Mathematics 105 (2000) 39--50.
\bibitem{MR1936948} G. Chen, R.J. Faudree, and M.S. Jacobson, Fragile graphs with small independent cuts, Journal of Graph Theory 41 (2002) 327--341.
\bibitem{chyu} G. Chen and X. Yu, A note on fragile graphs, Discrete Mathematics 249 (2002) 41--43.
\bibitem{ch} V. Chv\'{a}tal, Recognizing decomposable graphs, Journal of Graph Theory 8 (1984) 51--53.
\bibitem{codhgo} K. Coolsaet, S. D'hondt, and J. Goedgebeur, House of Graphs 2.0: A database of interesting graphs and more, Discrete Applied Mathematics 325 (2023) 97--107.
\bibitem{cofo} D.G. Corneil and J. Fonlupt, Stable set bonding in perfect graphs and parity graphs, Journal of Combinatorial Theory. Series B 59 (1993) 1--14.
\bibitem{klfi} S. Klein and C.M.H. de Figueiredo, The NP-completeness of multi-partite cutset testing, Congressus Numerantium 119 (1996) 217--222.
\bibitem{le_mosca_mueller_2008} V.B. Le, R. Mosca, and H. M\"{u}ller, On stable cutsets in claw-free graphs and planar graphs, Journal of Discrete Algorithms 6 (2008) 256--276.
\bibitem{krle} S. Kratsch and V.B. Le, On polynomial kernelization for Stable Cutset, arXiv 2407.02086.
\bibitem{lepf} V.B. Le and F. Pfender, Extremal graphs having no stable cutsets, Electronic Journal of Combinatorics 20 (2013) Paper 35, 7.
\bibitem{LeRa} V.B. Le and B. Randerath, On stable cutsets in line graphs, Theoretical Computer Science 301 (2003) 463--475.
\bibitem{MaOsRa} D. Marx, B. O'Sullivan, and I. Razgon, Finding small separators in linear time via treewidth reduction, ACM Transactions on Algorithms 9 (2013) Art. 30.
\bibitem{raraso} J. Rauch, D. Rautenbach, and U.S. Souza, Exact and parameterized algorithms for the independent cutset problem, Lecture Notes in Computer Science 14292 (2023) 378--391.
\bibitem{tu} A. Tucker, Coloring graphs with stable cutsets, Journal of Combinatorial Theory. Series B 34 (1983) 258--267.
\end{thebibliography}
\end{document}